\documentclass{l4dc2024}
\usepackage{wrapfig}
\usepackage{soul}
\usepackage[utf8]{inputenc}
\usepackage[T1]{fontenc}
\usepackage[ngerman,english]{babel}
\usepackage{enumitem}
\usepackage{graphicx}

\title[Computational Impact in Stochastic LQR]{Analyzing the Impact of Computation in Adaptive Dynamic Programming for Stochastic LQR Problem}
\usepackage{times}

\author{%
 \Name{Wenhan Cao}
 \Email{cwh19@mails.tsinghua.edu.cn}
 \\
  \addr
 School of Vehicle and Mobility, Tsinghua University, China
 \\
  \addr 
  Department of Computer Science, 
 The University of Manchester, UK
\AND
 \Name{Alexandre Capone} \Email{alexandre.capone@tum.de}\\
 \addr Chair of Information-oriented Control, Technical University of Munich, Germany
\AND
 \Name{Sandra Hirche} \Email{hirche@tum.de}\\
 \addr Chair of Information-oriented Control, Technical University of Munich, Germany
 \AND
 \Name{Wei Pan} \Email{wei.pan@manchester.ac.uk}\\
 \addr Department of Computer Science, 
 The University of Manchester, UK
}

\begin{document}

\maketitle

\begin{abstract}%
Adaptive dynamic programming (ADP) for stochastic linear quadratic regulation (LQR) demands the precise computation of stochastic integrals during policy iteration (PI). In a fully model-free problem setting, this computation can only be approximated by state samples collected at discrete time points using computational methods such as the canonical Euler-Maruyama method. Our research reveals a critical phenomenon: the sampling period can significantly impact control performance. This impact is due to the fact that computational errors introduced in each step of PI can significantly affect the algorithm’s convergence behavior, which in turn influences the resulting control policy. We draw a parallel between PI and Newton’s method applied to the Ricatti equation to elucidate how the computation impacts control. In this light, the computational error in each PI step manifests itself as an extra error term in each step of Newton’s method, with its upper bound proportional to the computational error. Furthermore, we demonstrate that the convergence rate for ADP in stochastic LQR problems using the Euler-Maruyama method is $O(h)$, with $h$ being the sampling period. A sensorimotor control task finally validates these theoretical findings.
\end{abstract}


\section{Introduction}\label{sec.intro}
In recent years, discrete-time reinforcement learning (RL) has proven effective in various domains, such as games \citep{silver2016mastering}, 
large language models \citep{OpenAI2023GPT4TR}, etc. Despite these advances, most systems, whether in natural sciences like physics \citep{einstein1905motion} and biology \citep{szekely2014stochastic}, in social sciences such as finance \citep{wang2020continuous} and psychology \citep{oravecz2011hierarchical}, or in engineering fields including robotics \citep{stager2016stochastic} and power systems \citep{milano2013systematic}, operate continuously in time and are inherently stochastic, governed by stochastic differential equations (SDEs). The complexity and continuous nature of these systems underscore the growing importance and necessity of using stochastic continuous-time RL approach, e.g., adaptive dynamic programming (ADP) algorithms \citep{jiang2011approximate, jiang2014adaptive, bian2016adaptive, bian2018stochastic, bian2019continuous, wei2023continuous}.

In practical applications, the model is often unknown, and accurately modeling and identifying SDEs is a complex task \citep{allen2007modeling,gray2011stochastic,browning2020identifiability}. For cases with \emph{unknown} system dynamics, policy iteration (PI)-based methods have been proposed for stochastic linear quadratic regulator (LQR) problems that involve state and control-dependent noise \citep{jiang2011approximate, jiang2014adaptive, bian2016adaptive, bian2018stochastic, bian2019continuous}. A significant challenge with this method is that each PI step involves the computation of a stochastic integral, as detailed in Equation (6) in \cite{bian2016adaptive} and also in \eqref{eq.PIa}. When the system is fully model-free, finding an exact analytical solution for this stochastic integral is not feasible due to the unknown analytical form of the integrand. For example, the analytical form of the integrand \(x^{\top}\otimes x^{\top}\) in the integral \(\int_t^{t+\delta_t} x^{\top}\otimes x^{\top} \rd s\) from \eqref{eq.PIa} is inaccessible since we cannot acquire the analytical form of the system state \(x\).

An alternative approach is to approximate the integral with discrete-time state samples using numerical methods, such as the canonical Euler-Maruyama method \citep{kloeden1992stochastic}. 
Consider an autonomous driving car equipped with various sensors such as LIDAR, radar, and cameras, which collect data with a sampling period of $h$.  By gathering these samples over a time span of $\delta_t = N_s h$, where $N_s$ is sample size, we can approximate the integral $\int_t^{t+\delta_t}x^{\top}\otimes x^{\top} \rd s$ in \eqref{eq.PIa} using the sensor data at points $x(t)$, $x(t+h)$, ..., $x(t+(N_s-1)h))$, see $\int_t^{t+\delta_t} x^{\top}\otimes x^{\top} \rd s \approx \sum_{i=0}^{N_s-1} x^{\top}(t+ih)\otimes x^{\top}(t+ih) \cdot h
$. In this scenario, the sampling period $h$ essentially serves as the time step in the Euler-Maruyama method to approximate the integral.

\begin{wrapfigure}{h}{0.49\textwidth}
    \centering
    \vspace{-0.5cm}
    \includegraphics[width=0.5\textwidth]{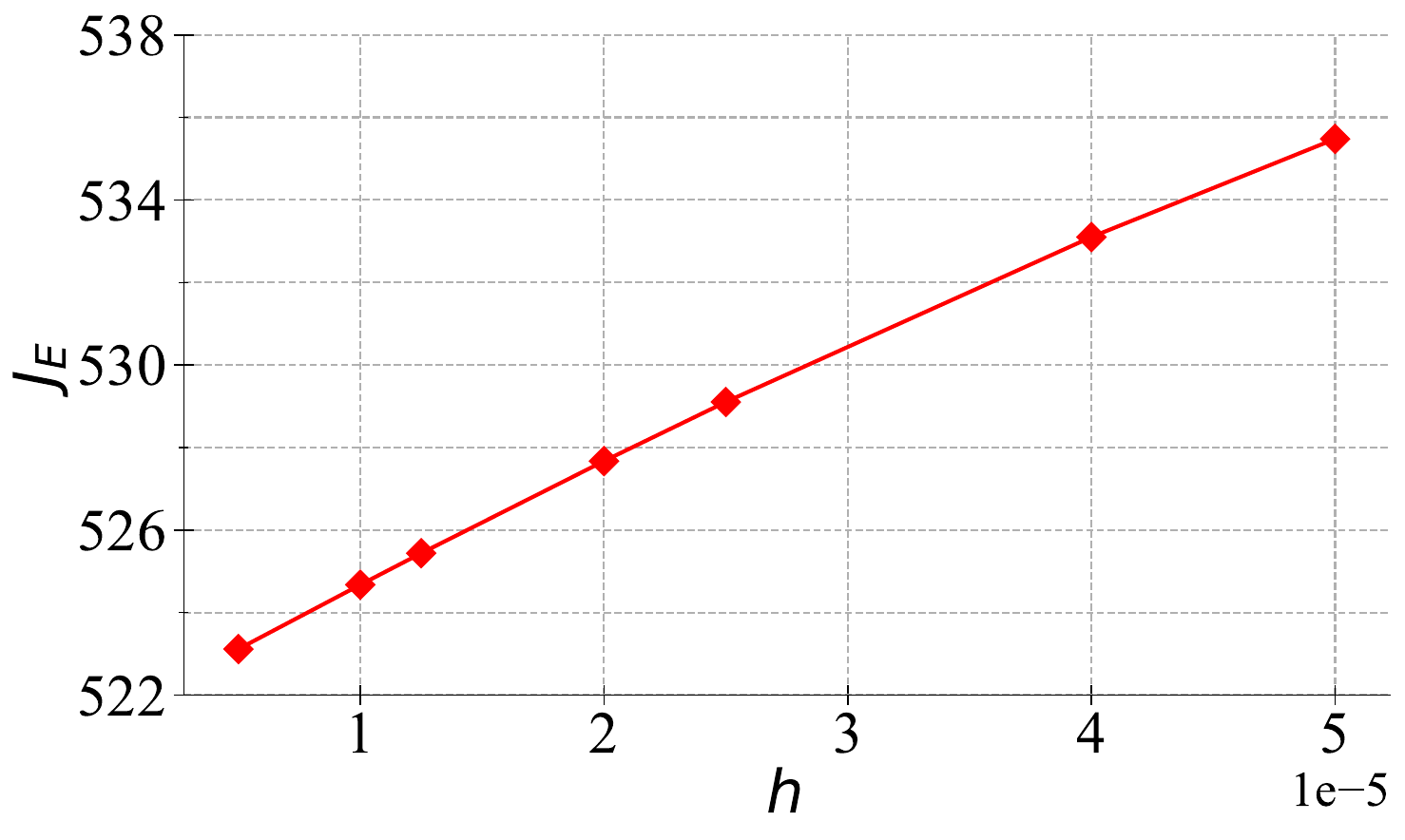}
    \vspace{-0.8cm}
    \caption{Expected cost $J_E$ for controllers, computed by the Euler-Maruyama method with varying sample period $h$ for a sensorimotor control task.}
    \vspace{-0.3cm}
    \label{fig.control performance}
\end{wrapfigure}

However, when applying the model-free ADP method to simulate a sensorimotor control task, as described in Section V of \cite{bian2016adaptive}, we observe that the sample period \( h \) in each interval can significantly impact the system's performance. As illustrated in Figure \ref{fig.control performance}, we find that as \( h \) decreases, the expected cost $J_E = \mathbb{E}\left\{\int_{0}^{\infty} \left(x^{\top}Qx + u^{\top}Ru \right) \rd t \right\}$ decreases significantly.
We analyze that the phenomenon ``computation impacts control'' can be attributed to the fact that a large sample period $h$ inevitably introduces computational errors when we approximate the stochastic integrals \eqref{eq.PIa} using samples collected at discrete time intervals. These computational errors can then lead to errors in each step of PI, accumulating throughout the entire PI process. Ultimately, this affects the convergence of ADP and consequently the control performance.
To understand how computation influences control in ADP for stochastic LQR problems, we first examine the solution error in each step of PI due to computational errors and then quantify the cumulative impact throughout the PI process. Our contributions can be summarized as follows:
\begin{itemize}
    \item We demonstrate that each step of PI necessitates the computation of stochastic integrals, which can only be approximated using samples collected at discrete time points. This approximation process inevitably introduces computational errors in each PI step. Taking advantage of the Euler-Maruyama method's convergence properties and analyzing the matrix equation's structure, we show that the solution error in each step of PI is proportional to the sampling period \( h \), in a system with a fixed sampling period.

    \item We prove that the PI proposed in \cite{jiang2011approximate, jiang2014adaptive, bian2016adaptive} can be interpreted as Newton's method for solving the generalized Riccati equation. When impacted by computational errors, the PI process resembles Newton's iteration with an additional error term. By leveraging the convergence properties of fixed-point iterations and assuming common Lipschitz conditions, we establish that the local convergence rate for ADP in stochastic LQR problems, in the presence of computational errors, is \( O(h) \).
\end{itemize}

The remainder of this paper is organized as follows:
In Section \ref{sec.problem setting}, we formulate the problem. Then, Section \ref{sec.computational error analysis} analyzes the computational error in the solution in each step of PI. Subsequently, Section \ref{sec.convergence analysis} provides a detailed convergence analysis. Finally, we present the numerical results in Section \ref{sec.results} and conclude our paper in Section \ref{sec.conclusion}.

\textbf{Notations:} \(\mathbb{R}^{m \times n}\) signifies the set of $m \times n$ real matrices. The symbol \(\|\cdot\|\) specifically refers to the Euclidean norm when applied to vectors. The Frobenius norm for matrices is denoted by \(\|\cdot\|_F\). We use \(\mathrm{vec}(\cdot)\) to indicate the vectorization of a matrix, transforming it into a column vector. The Kronecker product operation is represented by \(\otimes\). For matrix inequalities, \(P \succ 0\) (\(P \succeq 0\)) means that matrix \(P\) is positive (semi-)definite. The notation \(|\cdot|\) is used for element-wise absolute value when applied to matrices or vectors.

\section{Problem Formulation}\label{sec.problem setting}

Consider the linear system governed by the following SDE:

\begin{equation}\label{eq.SDE}
\rd x = (Ax + Bu)\rd t + B\rd w,    
\end{equation}
with 
\begin{equation}\label{eq.SDE noise}
\rd w = \sum_{i=1}^{q_1} F_i x \rd w_{1i} + \sum_{i=1}^{q_2} G_i u \rd w_{2i}. 
\end{equation}
Here, $\begin{bmatrix}
w_1, & w_2    
\end{bmatrix}^{\top} \in \mathbb{R}^{q_1+q_2}$ is the $\left\{\mathcal{F}_t \right\}$-adapted standard Brownian motion, where $\mathcal{F}_t$ is the $\sigma$-field generated by $\begin{bmatrix}
w_1, & w_2    
\end{bmatrix}^{\top}$. Since accurately modeling and identifying the SDE as \eqref{eq.SDE} is a complex task \citep{allen2007modeling,gray2011stochastic,browning2020identifiability}, we consider a \emph{model-free} problem setting, which means both $A \in \mathbb{R}^{n \times n}$, $B \in \mathbb{R}^{n \times m}$, $F_i \in \mathbb{R}^{m \times n}$ and $G_i \in \mathbb{R}^{m \times m}$ are \emph{unknown}. Given $x(0) = x_0 \in \mathbb{R}^n$, we consider the quadratic cost in LQR defined by

\begin{equation}\nonumber
J(x_0; u) = \mathbb{E}\left\{\int_{0}^{\infty} \left(x^{\top}Qx + u^{\top}Ru \right) \rd t | x_0 \right\},    
\end{equation}
where $Q \succeq 0$ and $R \succ 0$. Suppose $\begin{bmatrix}
A & BF_1 & BF_2 & ... & BF_{q_1} | \sqrt{Q}  
\end{bmatrix}$ is  observable \citep{jiang2011approximate,jiang2014adaptive,bian2016adaptive,bian2018stochastic,bian2019continuous}, the objective is to find the controller $u$ which minimizes $J(x_0, u)$.

\begin{definition}
[Admissible matrix]
A matrix $K \in \mathbb{R}^{m \times n}$ is called admissible if the system described by \eqref{eq.SDE} and \eqref{eq.SDE noise} with control policy $u = -Kx$ is globally asymptotically stable (GAS) at the origin in the mean square sense \citep{willems1976feedback} and $J(x_0, u) < \infty, \forall x_0 \in \mathbb{R}^n$.  
\end{definition}

\begin{assumption}[Admissible initial matrix]
The system described by \eqref{eq.SDE} and \eqref{eq.SDE noise} is mean square stabilizable with a known admissible $K_0$ \citep{jiang2011approximate,jiang2014adaptive,bian2016adaptive,bian2018stochastic,bian2019continuous}.  
\end{assumption}

It is shown in \cite{mclane1971optimal} that $J(x_0; u)$ is optimized under the optimal controller $u^* = -K^* x$, where $K^* = \left(\Sigma_{P^*} + R\right)^{-1}B^{\top}P^*$ and $P^* \in \mathbb{R}^{n \times n}$ is the unique symmetric positive definite solution to the following generalized algebraic Riccati equation

\begin{subequations}\label{eqs.ricatti}
\begin{equation}\label{eq.ricatti}
\begin{aligned}
\mathcal{T}_P = 0,
\end{aligned}
\end{equation}
with
\begin{equation}\label{eq.ricatti definition}
\begin{aligned}
\mathcal{T}_P &:= A^{\top}P + PA - PB(\Sigma_{P} + R)^{-1}B^{\top}P + \Pi_P + Q,
\\
\Pi_P &:= \Sigma_{i=1}^{q_1} F_i^{\top}B^{\top}PBF_i, \;\quad  \Sigma_{P} := \Sigma_{i=1}^{q_2} G_i^{\top}B^{\top}PBG_i.    
\end{aligned}
\end{equation}    
\end{subequations}

In situations where the system models are \emph{fully unknown}, the data-driven PI approach can be utilized to solve the Riccati equation \citep{jiang2011approximate, jiang2014adaptive, bian2016adaptive, bian2018stochastic, bian2019continuous}. In the \(k\)-th step of PI, the control policy \(u_k = -K_k x + e_k\), where \(e_k\) represents the exploration noise, generates the data. The \(k\)-th step of PI can specifically be expressed as:
\begin{subequations}\label{eqs.PI}
\begin{align}
\label{eq.PIa}
\begin{split}
x^{\top} \otimes x^{\top} |_{t}^{t+\delta_t} \mathrm{vec}(P_k) =& 2 \left( \int_{t}^{t+\delta_t} x^{\top} \otimes \rd \hat{w}_k^{\top} \right) \mathrm{vec} (B^{\top}P_k) \\
&+ \left( \int_{t}^{t+\delta_t} u_k^{\top} \otimes u_k^{\top} \rd s \right) \mathrm{vec} (\Sigma_{P_k}) \\
&- \left( \int_{t}^{t+\delta_t} x^{\top} \otimes x^{\top} \rd s \right) \left(K_k^{\top} \otimes K_k^{\top} \right) \mathrm{vec}(\Sigma_{P_k}) \\
& -\left( \int_{t}^{t+\delta_t} x^{\top} \otimes x^{\top} \rd s \right) \mathrm{vec} (Q + K_k^{\top}RK_k),
\end{split} \\
\label{eq.PIb}
K_{k+1} =& \left(\Sigma_{P_k} + R\right)^{-1}B^{\top}P_k.
\end{align}
\end{subequations}
Here, $\rd \hat{w}_k = e_k \rd s + \sum_{i=1}^{q_1} F_i x \rd w_{1i} + \sum_{i=1}^{q_2} G_i u \rd w_{2i}$ is assumed known, as it can be reconstructed through noisy control inputs and the state samples \citep{jiang2011approximate,jiang2014adaptive,bian2016adaptive,bian2018stochastic,bian2019continuous}.

\begin{remark}\label{remark.Kleinman}
The matrix \( P_k \) in \eqref{eq.PIa} is the solution to the Lyapunov equation
\begin{equation}\label{eq.PI aa}
(A-BK_k)^{\top}P_k + P_k(A-BK_k) + Q + \Pi_{P_k} + K_k^{\top}\left(\Sigma_{P_k}+R\right)K_k = 0,
\end{equation}
where \eqref{eq.PIa} can be regarded as the data-driven counterpart of \eqref{eq.PI aa}. Throughout the PI process, the optimal solution of the Riccati equation \( P^* \) is eventually obtained, i.e., \(\lim_{k\to\infty} P_k = P^*\) (see Theorem 1 in \cite{kleinman1969optimal} and Theorem 1 in \cite{bian2016adaptive}).
\end{remark}

As shown in Remark \ref{remark.Kleinman}, when the system model is known, or when there is no computation error—meaning the stochastic integrals can be perfectly approximated by discrete-time state samples—the PI process in \eqref{eqs.PI} eventually converges to \( P^* \) and \( K^* \). However, as we have discussed in Section \ref{sec.intro}, computational error is inevitable since it is impossible to sample with an infinitesimally small period. In this context, studying the impact of computational error on the algorithm's convergence, specifically on the solution's accuracy at convergence, becomes crucial. This brings up three key questions:
\begin{enumerate}[label=\textbf{Q\arabic*:}]
\item How can we quantify the computational error of the integral in \eqref{eq.PIa}? 
\item  How does this computational error affect the solution of PI in each step?
\item How does the computational error influence the overall convergence of ADP algorithm?
\end{enumerate}

\section{Computational Error Analysis}\label{sec.computational error analysis}

To answer \textbf{Q1}, we examine the computational error of the stochastic integral when applying the Euler-Maruyama method. Consider the stochastic integral \( X(t) = X(0) + \int_{0}^{t} \mu(X(s))\rd s + \int_{0}^{t} \sigma(X(s))\rd W(s) \), where \( W(s) \) represents the standard Brownian motion. The following lemma can quantify the computational error of this integral.

\begin{lemma}[Computational error of Euler-Maruyama
method \citep{bally1995euler,bally1996law1,bally1996law2}]\label{lemma.Euler computational error}
Suppose $X^{h}(t)$ is the discrete-time approximation of $X(t)$ through the Euler-Maruyama method with the time step $h$. If $\mu$ and $\sigma$ are bounded and Lipschitz continuous, then the computational error $|\mathbb{E}X(t) - \mathbb{E}X^{h}(t)|$ satisfies
\begin{equation}
|\mathbb{E}X(t) - \mathbb{E}X^{h}(t)| \leq Ch, 
\end{equation}
with $C>0$ being a constant number depending on the formulation of the integral.
\end{lemma}

Next, we will address \textbf{Q2}, i.e., examining how the computational error in the stochastic integral in \eqref{eq.PIa} affects the solution in each step of PI. According to \cite{jiang2011approximate,jiang2014adaptive,bian2016adaptive}, PI  cannot be directly solved by \eqref{eq.PIa}. Instead, \eqref{eq.PIa} can be formulated in a matrix equation \eqref{eq.PI matrix} by collecting the data in $l$ time intervals $[t_0, t_1], [t_1, t_2], ..., [t_{l-1}, t_l]$:
\begin{equation}\label{eq.PI matrix}
\begin{aligned}
\Theta_k
\begin{bmatrix}
\mathrm{vec}(P_k) \\
\mathrm{vec}(B^{\top}P_k) \\
\mathrm{vec}(\Sigma_{P_k})
\end{bmatrix} = \Xi_k,
\end{aligned}
\end{equation}
with
\begin{equation}\label{eq.integrals matrix}
\begin{aligned}
\Theta_k &= \begin{bmatrix}
\delta_{xx}^{k}, & -2I_{x\hat{w}}^{k}, & I_{xx}^k\left(K_k^{\top} \otimes K_k^{\top}\right) - I_{uu}^k    
\end{bmatrix},
\\
\Xi_k &= -I_{xx}^k \mathrm{vec}(Q+ K_k^{\top}RK_k),
\\
\delta^{k}_{xx} &= \begin{bmatrix}
x \otimes x|^{t_1,k}_{t_0,k}, &
x \otimes x|^{t_2,k}_{t_1,k}, &
\hdots, &
x \otimes x|^{t_{l},k}_{t_{l-1},k}
\end{bmatrix}^{\top},    
\\
I^{k}_{xx} &= 
\begin{bmatrix}
\int_{t_{0,k}}^{t_{1,k}} x \otimes x \rd s, &
\int_{t_{1,k}}^{t_{2,k}} x \otimes x \rd s, &
\hdots, &
\int_{t_{l-1,k}}^{t_{l,k}} x \otimes x \rd s
\end{bmatrix}^{\top},
\\
I^{k}_{uu} &= 
\begin{bmatrix}
\int_{t_{0,k}}^{t_{1,k}} u_k \otimes u_k \rd s, &
\int_{t_{1,k}}^{t_{2,k}} u_k \otimes u_k \rd s, &
\hdots, &
\int_{t_{l-1,k}}^{t_{l,k}} u_k \otimes u_k \rd s
\end{bmatrix}^{\top},
\\
I^{k}_{x\hat{w}} &= 
\begin{bmatrix}
\int_{t_{0,k}}^{t_{1,k}} x \otimes \rd \hat{w}_k, &
\int_{t_{1,k}}^{t_{2,k}} x \otimes \rd \hat{w}_k, &
\hdots, &
\int_{t_{l-1,k}}^{t_{l,k}} x \otimes \rd \hat{w}_k
\end{bmatrix}^{\top}.
\end{aligned}
\end{equation}
For simplicity, we denote \( s_k := \begin{bmatrix} \mathrm{vec}^{\top}(P_k),& \mathrm{vec}^{\top}(B^{\top}P_k), & \mathrm{vec}^{\top}(\Sigma_{P_k}) \end{bmatrix}^{\top} \). In the matrix equation \eqref{eq.PI matrix}, both \( \Theta_k \) and \( \Xi_k \) are stochastic. 
By taking expectations on both sides of \eqref{eq.PI matrix}, we obtain
\begin{equation}\label{eq.PI matrix expectation}
\mathbb{E}[\Theta_k] s_k = \mathbb{E}[\Xi_k],    
\end{equation}
since \( s_k \) is a deterministic vector. Under the law of large numbers \citep{grimmett2020probability}, we should gather a significant amount of data and compute the average to approximate \( \mathbb{E}[\Theta_k] \) and \( \mathbb{E}[\Xi_k] \) using the data \citep{jiang2011approximate,jiang2014adaptive,bian2016adaptive}.


\begin{remark}
The uniqueness of the solution for equation \eqref{eq.PI matrix} hinges on the rank condition:
$$
\mathrm{rank} \begin{bmatrix}
I_{xx}^k & I_{x\hat{w}}^k & I_{uu}^k    
\end{bmatrix} = \frac{n(n+1)}{2} + mn + \frac{m(m+1)}{2}.
$$
This condition is verified in \cite{bian2016adaptive} and can be seen as the persistent excitation condition in adaptive control. A bounded exploration noise \( e_k \), such as sinusoidal or random signals, can be introduced to the input during learning to meet the rank condition. Given that \( K_k \) is admissible for all \( k = 0, 1, \ldots \), the solution of system \eqref{eq.SDE} equipped with the controller \( u_k \) is guaranteed to not escape in finite time with a probability of one \citep{bian2016adaptive}.
\end{remark}

As mentioned in Section \ref{sec.intro}, the integrals in matrices $\Theta_k$ and $\Xi_k$ cannot be obtained analytically since the integrand is not analytically accessible in a model-free setting. Instead, we opt to approximate these integrals using samples collected at discrete time points using the canonical Euler-Maruyama method. Denote the approximated matrix acquired by the Euler-Maruyama method as $\mathbb{E}[\hat{\Theta}_k]$ and $\mathbb{E}[\hat{\Xi}_k]$, thus the matrix equation we actually solve is 
\begin{equation}\label{eq.PI matrix expectation Euler}
\mathbb{E}[\hat{\Theta}_k] \hat{s}_k = \mathbb{E}[\hat{\Xi}_k].    
\end{equation}
Here, \( \hat{s}_k \) is the solution of PI obtained by using the Euler-Maruyama method to approximate the data matrices. The following theorem establishes the error bound between the idealized solution \( s_k \), obtained by \eqref{eq.PI matrix expectation} without computational error, and \( \hat{s}_k \).

\begin{theorem}\label{theorem.computational error per iter}
If $|s_k | \leq b$, where $b > 0$ is a constant number, and $l = n^2 + m^2 + mn$, then we have $|s_k - \hat{s}_k| \leq C_s h$, where $C_s \in \mathbb{R}^{l}$ is a constant vector with all elements being positive.
\end{theorem}
\begin{proof}
According to Lemma \ref{lemma.Euler computational error}, if approximated by the Euler-Maruyama method, the computational error for all the elements in the integrals $I_{xx}^k$, $I_{uu}^k$ and $I_{x\hat{w}}^k$ in \eqref{eq.PI matrix expectation} satisfies $O(h)$. Thus, we have 
$\left|\mathbb{E}[\Theta_k] - \mathbb{E}[\hat{\Theta}_k]\right| \leq C_{\Theta} h$ and $\left|\mathbb{E}[\Xi_k] - \mathbb{E}[\hat{\Xi}_k]\right| \leq C_{\Xi} h$ , where $C_{\Theta}$ is an ${l \times l}$ matrix and $C_{\zeta}$ is an ${l}$ dimension vector with all elements being positive. For $l = n^2 + m^2 + mn$, $\Theta_k$ and $\hat{\Theta}_k$ are square matrices, thus we have
\begin{equation}\nonumber
\begin{aligned}
\left|s_k - \hat{s}_k\right| &= \left| \left(\mathbb{E}[\hat{\Theta}_k]\right)^{-1}\mathbb{E}[\hat{\Theta}_k] (s_k -\hat{s}_k)\right|
\leq \left|\left(\mathbb{E}[\hat{\Theta}_k]\right)^{-1} \right| \left|\mathbb{E}[\hat{\Theta}_k] (s_k -\hat{s}_k) \right|.
\end{aligned}
\end{equation}

Since $\hat{s}_k, s_k$ satisfies $\hat{\Xi}_k - \mathbb{E}[\hat{\Theta}_k] \hat{s}_k = {\Xi}_k - \mathbb{E}[{\Theta}_k] {s}_k = 0$, we obtain
\begin{equation}\label{eq.proof}
\begin{aligned}
\left|s_k - \hat{s}_k\right|
&\leq \left|\left(\mathbb{E}[\hat{\Theta}_k]\right)^{-1} \right| \left|\mathbb{E}[\hat{\Theta}_k]s_k  - \mathbb{E}[\hat{\Theta}_k] \hat{s}_k - \hat{\Xi}_k + \hat{\Xi}_k \right|
\\
&\leq \left|\left(\mathbb{E}[\hat{\Theta}_k]\right)^{-1} \right| \left|\mathbb{E}[\hat{\Theta}_k]s_k - \hat{\Xi}_k - \left(\mathbb{E}[\hat{\Theta}_k] \hat{s}_k - \hat{\Xi}_k \right) \right|
\\
&= \left|\left(\mathbb{E}[\hat{\Theta}_k]\right)^{-1} \right| \left|\mathbb{E}[\hat{\Theta}_k]s_k - \hat{\Xi}_k - \left(\mathbb{E}[{\Theta}_k] {s}_k - {\Xi}_k \right) \right|
\\
&\leq \left|\left(\mathbb{E}[\hat{\Theta}_k]\right)^{-1} \right| \left(\left|\mathbb{E}[\hat{\Theta}_k]s_k - \mathbb{E}[{\Theta}_k] {s}_k \right| + \left|{\Xi}_k  - \hat{\Xi}_k \right|\right)
\\
&\leq 
\left|\left(\mathbb{E}[\hat{\Theta}_k]\right)^{-1} \right| \left(
b C_{\Theta} \mathbf{1}_l  + C_{\Xi}
\right) \cdot h.
\end{aligned}    
\end{equation}
Here, \( \mathbf{1}_l \in \mathbb{R}^{l }\) represents a vector whose elements are all ones. Note that the last row of \eqref{eq.proof} is based on the fact that $|AB| \leq |A||B|$ for matrices $A$ and $B$ and $|a| \leq \|a\|\mathbf{1}$ for vector $a$. Since \( \mathbb{E}[\hat{\Theta}_k] \), \( C_{\Theta} \), and \( C_{\Xi} \) are constants independent of \( h \), we define \( \left|\left(\mathbb{E}[\hat{\Theta}_k]\right)^{-1} \right| \left( b C_{\Theta}\mathbf{1}_l  + C_{\Xi} \right) \) as \( C_s \) with all elements being positive and obtain $|s_k - \hat{s}_k| \leq C_s h$.
\end{proof}

\section{Convergence Analysis}\label{sec.convergence analysis}

To answer \textbf{Q3}, we first formulate the PI as Newton's method to solve the Ricatti equation, as shown in the subsequent lemma:
\begin{lemma}[PI as Newton's method]\label{lemma.Newton's method}
Iteratively performing \eqref{eq.PI matrix expectation} and \eqref{eq.PIb} equals Newton's method to solve the Ricatti equation \eqref{eqs.ricatti}. Specifically, the iteration of Newton's method can be written as 
\begin{equation}\label{eq.Newton's method original}
P_{k+1} = P_k - \left(\mathcal{T}'_{P_k}\right)^{-1} \mathcal{T}_{P_k}.     
\end{equation}
Here, $\mathcal{T}'_{P}$ is the Fréchet differential of $\mathcal{T}_P$ at $P$ where $\mathcal{T}_P$ is defined as in \eqref{eq.ricatti definition}.
\end{lemma}

\begin{proof}
$\forall P \succ 0$ and $W \succ 0$, we have
\begin{equation}\label{eq.delta T}
\begin{aligned}
&\mathcal{T}_{P+\epsilon W} - \mathcal{T}_P 
\\
=& A^{\top}(P+\epsilon W) + (P+\epsilon W)A - (P+\epsilon W)B\left(\Sigma_{P+\epsilon W} + R\right)^{-1}B^{\top}(P+\epsilon W) + \Pi_{P+\epsilon W} + Q 
\\
-& \left(A^{\top}P + PA - PB\left(\Sigma_{P} + R\right)^{-1}B^{\top}P + \Pi_P + Q\right).
\end{aligned}  
\end{equation}
By using Woodbury matrix identity \citep{hager1989updating}, we obtain
\begin{equation}\label{eq.matrix inverse lemma}
\begin{aligned}
\left(\Sigma_{P+\epsilon W} + R\right)^{-1} &= \left(\Sigma_{P} + R+\Sigma_{\epsilon W}\right)^{-1} 
\\
&= \left(\Sigma_{P} + R\right)^{-1} - \left((\Sigma_{P} + R) + \frac{1}{\epsilon}(\Sigma_{P} + R) \Sigma_{W}^{-1} (\Sigma_{P} + R) \right)^{-1}.  
\end{aligned}
\end{equation}
Combing \eqref{eq.matrix inverse lemma} with \eqref{eq.delta T}, we have
\begin{equation}\nonumber
\begin{aligned}
&\mathcal{T}_{P+\epsilon W} - \mathcal{T}_P 
\\
=& \epsilon A^{\top}W + \epsilon WA + \Pi_{\epsilon W} - \epsilon WB\left(\Sigma_{P} + R\right)^{-1}B^{\top}P - \epsilon PB\left(\Sigma_{P} + R\right)^{-1}B^{\top}W
\\
&+(P+\epsilon W)B\left((\Sigma_{P} + R) + \frac{1}{\epsilon}(\Sigma_{P} + R) \Sigma_{W}^{-1} (\Sigma_{P} + R) \right)^{-1}B^{\top}(P+\epsilon W)
\\
&-\epsilon^2WB\left(\Sigma_{P} + R\right)^{-1}B^{\top}W.
\end{aligned}  
\end{equation}
Thus, the Gâteaux and Fréchet differential \citep{kantorovich2016functional,yosida2012functional} of $\mathcal{T}_P$ at $P$ can be expressed as
\begin{equation}\label{eq.F derivative}
\begin{aligned}
\mathcal{T}'_P W =& \lim_{\epsilon\to 0} \frac{\mathcal{T}_{P+\epsilon W}-\mathcal{T}_P}{\epsilon}
\\
=& A^{\top}W + WA + \Pi_W - WB\left(\Sigma_{P} + R\right)^{-1}B^{\top}P - PB\left(\Sigma_{P} + R\right)^{-1}B^{\top}W
\\
&+PB(\Sigma_{P} + R)^{-1} \Sigma_{W} (\Sigma_{P} + R)^{-1}B^{\top}P.
\end{aligned}  
\end{equation}
To find the root of Ricatti equation \eqref{eqs.ricatti}, Newton's method produces successively better approximations of the root. 
Newton's iteration in \eqref{eq.Newton's method original} equals to
\begin{equation}\label{eq.Newton's method}
\mathcal{T}'_{P_k}P_{k+1} - \mathcal{T}'_{P_k}P_k  + \mathcal{T}_{P_k} = 0.  
\end{equation}
By substituting $\mathcal{T}'_{P_k}$ from \eqref{eq.F derivative} into \eqref{eq.Newton's method}, we have
\begin{equation}\label{eq.Newton's method2}
\begin{aligned}
&A^{\top}P_{k+1} + P_{k+1}A + \Pi_{P_{k+1}} - P_{k+1}B\left(\Sigma_{P_k} + R\right)^{-1}B^{\top}P_k - P_kB\left(\Sigma_{P_k} + R\right)^{-1}B^{\top}P_{k+1}
\\
+&P_kB(\Sigma_{P_k} + R)^{-1} \Sigma_{P_{k+1}} (\Sigma_{P_k} + R)^{-1}B^{\top}P_k  -A^{\top}P_k-P_kA-\Pi_{P_k} 
\\
+& 2P_kB\left(\Sigma_{P_k} + R\right)^{-1}B^{\top}P_k - P_kB\left(\Sigma_{P_k} + R\right)^{-1}\Sigma_{P_k}\left(\Sigma_{P_k} + R\right)^{-1}B^{\top}P_k 
\\
+& A^{\top}P_k + P_kA - P_kB\left(\Sigma_{P_k} + R\right)^{-1}B^{\top}P_k +\Pi_{P_k} + Q = 0.  
\\
& \ \ \ \ \ \ \ \ \ \ \ \ \ \ \ \ \ \ \ \ \ \ \ \ \ \ \ \ \ \ \ \ \ \ \ \ \ \ \ \ \ \ \ \ \ \ \ \ \ \ \ \ \ \ \ \ \  \Updownarrow
\\
&A^{\top}P_{k+1} + P_{k+1}A + \Pi_{P_{k+1}} - P_{k+1}B\left(\Sigma_{P_k} + R\right)^{-1}B^{\top}P_k - P_kB\left(\Sigma_{P_k} + R\right)^{-1}B^{\top}P_{k+1}
\\
+&P_kB(\Sigma_{P_k} + R)^{-1} \Sigma_{P_{k+1}} (\Sigma_{P_k} + R)^{-1}B^{\top}P_k
\\
+& P_kB\left(\Sigma_{P_k} + R\right)^{-1}B^{\top}P_k - P_kB\left(\Sigma_{P_k} + R\right)^{-1}\Sigma_{P_k}\left(\Sigma_{P_k} + R\right)^{-1}B^{\top}P_k + Q = 0.
\end{aligned} 
\end{equation}
By noticing that
\begin{equation}\nonumber
\begin{aligned}
&P_kB\left(\Sigma_{P_k} + R\right)^{-1}B^{\top}P_k - P_kB\left(\Sigma_{P_k} + R\right)^{-1}\Sigma_{P_k}\left(\Sigma_{P_k} + R\right)^{-1}B^{\top}P_k 
\\
=&
P_kB\left(\Sigma_{P_k} + R\right)^{-1}R\left(\Sigma_{P_k} + R\right)^{-1}B^{\top}P_k, 
\end{aligned}  
\end{equation}
\eqref{eq.Newton's method2} can be transformed into:
\begin{equation}\label{eq.Newton's method3}
\begin{aligned}
&A^{\top}P_{k+1} -P_kB\left(\Sigma_{P_k}+R\right)^{-1}B^{\top}P_{k+1} + P_{k+1}A -P_{k+1}B\left(\Sigma_{P_k}+R\right)^{-1} B^{\top}P_k
\\
+&Q+\Pi_{P_{k+1}} + P_kB\left(\Sigma_{P_k}+R\right)^{-1}\left(\Sigma_{P_{k+1}} +R
\right)\left(\Sigma_{P_k}+R\right)^{-1} B^{\top}P_k = 0.
\end{aligned} 
\end{equation}

As for PI, iteratively applying \eqref{eq.PI matrix expectation} and solving for the $k$-th step control gain matrix \( K_k \) using \eqref{eq.PIb} is equivalent to iteratively performing \eqref{eq.PI aa} and \eqref{eq.PIb}. By combining \eqref{eq.PI aa} with \( K_k \) defined in \eqref{eq.PIb}, we can obtain \eqref{eq.Newton's method3}. This implies that each step of PI in \eqref{eqs.PI} inherently executes one step of Newton's method for the Riccati equation \eqref{eqs.ricatti}.
\end{proof}

According to Lemma \ref{lemma.Newton's method}, in the absence of computational error, the PI defined by \eqref{eq.PI matrix expectation} and \eqref{eq.PIb} is precisely Newton's method for solving the Riccati equation \eqref{eqs.ricatti}. However, as demonstrated in Theorem \ref{theorem.computational error per iter}, computational errors do exist in each step of PI, introducing an extra error term \( E_k \) to the original Newton's method in \eqref{eq.Newton's method original}:
\begin{equation}\label{eq.Newton's method with error}
\hat{P}_{k+1} = \hat{P}_k - (\mathcal{T}'_{\hat{P}_k})^{-1} \mathcal{T}_{\hat{P}_k} + E_k = \Psi_{\hat{P}_k} + E_k.   
\end{equation}
Here, the introduction of the extra error term \( E_k \) results from the computational error in the integral. To distinguish this iteration \eqref{eq.Newton's method with error} with computational error from the original Newton iteration \eqref{eq.Newton's method original}, we denote the result of the iteration in \eqref{eq.Newton's method with error} as \(\hat{P}\). Additionally, for simplicity, we define the operator \(\Psi_P := P - (\mathcal{T}_P^{'})^{-1}\mathcal{T}_P\). The following lemma provides a natural approach to analyze the convergence of the numerical iteration in \eqref{eq.Newton's method with error}.

\begin{lemma}[Error Bound of Fixed Point Iteration \cite{urabe1956convergence}]\label{lemma.convergence of FPI}    
Suppose the operator \(\Psi\) defined on the set \(\mathbb{R}^{n\times n}\) satisfies: (i) \textbf{Lipschitz condition}: \(\|\Psi_P - \Psi_{P'}\|_F \leq L_P \|P-P'\|_F\) with \(L_P < 1\); (ii) \textbf{Feasibility of the initial step of iteration}: For \(P_0 \in \mathbb{R}^{n\times n}\), \(\Psi_{P_0} \in \mathbb{R}^{n\times n}\); (iii) \textbf{Sphere containment}: The sphere \(\left\{h: \|h - P_1\|_F \leq \frac{L_P}{1-L_P} \| P_1 - P_0\|_F \right\}\) is contained in \(\mathbb{R}^{n\times n}\). Then, for the iteration \(\hat{P}_{k+1} = \Psi_{\hat{P}_k} + E_k\) in \eqref{eq.Newton's method with error} with \(\|\hat{E}_k \|_F \leq \Delta\) and the iteration \(P_{k+1} = \Psi_{P_k}\), if \(P_0 = \hat{P}_0\) is satisfied, we have
\begin{equation}\label{eq.convergence error}
\|P_k - \hat{P}_k\|_F < \frac{\Delta}{1-L_P}.   
\end{equation}
\end{lemma}

Lemma \ref{lemma.convergence of FPI} indicates that the extra error term's upper bound $\Delta$ constrains the discrepancy between solutions of Newton's method with and without computational error. Combine Lemma \ref{lemma.Newton's method}, Lemma \ref{lemma.convergence of FPI} with Theorem \ref{theorem.computational error per iter}, we can conclude with the following theorem.
\begin{theorem}[Convergence of PI with Computational Error]\label{theorem.convergence of PI}
Suppose the conditions in Theorem \ref{theorem.computational error per iter} and Lemma \ref{lemma.convergence of FPI} are satisfied, then we have 
\begin{equation}\label{eq.final result}
\lim_{k \to \infty}\| \hat{P}_{k} - P^* \|_F \leq \frac{\|C_s\|}{1-L_P}h.
\end{equation}
Here, $L_P$ is the Lipschitz constant in Lemma \ref{lemma.convergence of FPI} and $C_s$ is the constant vector defined in Theorem \ref{theorem.computational error per iter}.
\end{theorem}

\begin{proof}
According to Theorem \ref{theorem.computational error per iter}, we have $|s_k - \hat{s}_k | \leq C_s h$, where $C_s$ is the constant vector with all elements positive. For $P_k = \hat{P}_k$, we have $\|P_{k+1} - \hat{P}_{k+1} \|_F = \|\mathrm{vec}(P_{k+1}) - \mathrm{vec}{\hat{P}_{k+1}}\| \leq \|s_{k+1} - \hat{s}_{k+1} \| \leq \|C_s\| h.$ Thus, we obtain $\| E_k \|_F \leq \|C_s\| h := \Delta$. According to \eqref{eq.convergence error} in Lemma \ref{lemma.convergence of FPI}, we have 
\begin{equation}\label{eq.P_k - P_hat_k}
\| P_k - \hat{P}_k\|_F < \frac{\|C_s\| }{1-L_P}h.
\end{equation}
According to the convergence of PI without computational errors, we have $\lim_{k \to \infty} P_k = P^*$ as shown in Remark \ref{remark.Kleinman}. Taking the limit on both sides of \eqref{eq.P_k - P_hat_k}, we can obtain \eqref{eq.final result}.
\end{proof}

\begin{remark}
Theorem \ref{theorem.convergence of PI} demonstrates that under the conditions of Theorem \ref{theorem.computational error per iter} and Lemma \ref{lemma.convergence of FPI}, the convergence rate of ADP for stochastic LQR problems using the Euler-Maruyama method is $O(h)$. However, the Lipschitz bound condition in Lemma \ref{lemma.convergence of FPI} can sometimes be challenging to verify. In such cases, alternative forms of convergence analysis for Newton's method can be considered. These alternatives might impose less stringent conditions on the operator \(\Psi\) but limit the convergence results to a more local scope. For instance, the result in \cite{lancaster1966error} can be used to develop convergence analysis with different conditions, yet maintaining the same convergence rate of \(O(h)\).
\end{remark}

\begin{remark}
For the initial state $x(0)$ with zero mean and variance $X_0 = \mathbb{E}\{x(0) x(0)^\top\}$, the expected cost $J_E = \mathbb{E}\left\{\int_{0}^{\infty} (x^\top Qx + u^\top Ru) \rd t \right\} = \mathrm{Tr}(PX_0)$. Thus, the convergence rate $O(h)$ also applies to the expected cost $J_E$. This is why Figure \ref{fig.control performance} shows a straight line, and the intercept on the y-axis can be considered as the optimal expected cost $J_E^*$.
\end{remark}

\section{Numerical Results}\label{sec.results}
To validate the results in Theorem \ref{theorem.convergence of PI}, we consider human arm motion dynamics in a force field dependent on velocity \citep{bian2016adaptive,liu2007evidence}:   
\begin{equation}\nonumber
\begin{aligned}
\rd p = v \rd t,\;\;
m \rd v = (a - bv + f) \rd t,\;\;
\tau \rd a = (u-a) \rd t + \rd w,
\end{aligned} 
\end{equation}
where \( p = \begin{bmatrix}
p_x, & p_y    
\end{bmatrix}^{\top} \) represents the hand position, \( v = \begin{bmatrix}
v_x, & v_y    
\end{bmatrix}^{\top} \) denotes the hand velocity, and \( a = \begin{bmatrix}
a_x, & a_y    
\end{bmatrix}^{\top} \) indicates the acceleration. Additionally, \( u \) is the control input, \( m \) is the hand mass, \( b \) is the viscosity constant, \( \tau \) is the time constant, and \( \mathrm{d}w \) is the control-dependent noise satisfying
\begin{equation}\nonumber
\rd w = \begin{bmatrix}
c_1 & 0 \\ c_2 & 0    
\end{bmatrix}\begin{bmatrix}
u_x \\ u_y    
\end{bmatrix} \rd w_1 + \begin{bmatrix}
0 & c_2 \\ 0  & c_1    
\end{bmatrix}\begin{bmatrix}
u_x \\ u_y    
\end{bmatrix} \rd w_2.    
\end{equation}
Here, $\begin{bmatrix}
w_1, & w_2    
\end{bmatrix}$ are independent standard Wiener processes with noise magnitudes $c_1$ and $c_2$. The force $f$ is dependent on the velocity $v_x$ and $v_y$. Our simulation setup is consistent with that described by \cite{bian2016adaptive}, and the detailed parameter configurations can be found there due to space limit.

We plot the solution of PI, obtained by iteratively applying the operations in \eqref{eq.PI matrix expectation} and \eqref{eq.PIb}. The stochastic integrals in \eqref{eq.PI matrix expectation} are approximated using state samples collected at discrete times with varying sampling periods \(h\), via the Euler-Maruyama method. Figure \ref{fig.pnormandknorm} shows the Frobenius norm of the error in the resulting value matrix \(\hat{P}_{\infty}\) relative to the optimal value matrix \(P^*\), and the corresponding control gain matrix \(\hat{K}_{\infty}\) compared to the optimal control gain matrix \(K^*\).

\begin{figure}[htbp]
    \centering
    \begin{minipage}{0.49\textwidth}
        \includegraphics[width=\textwidth]{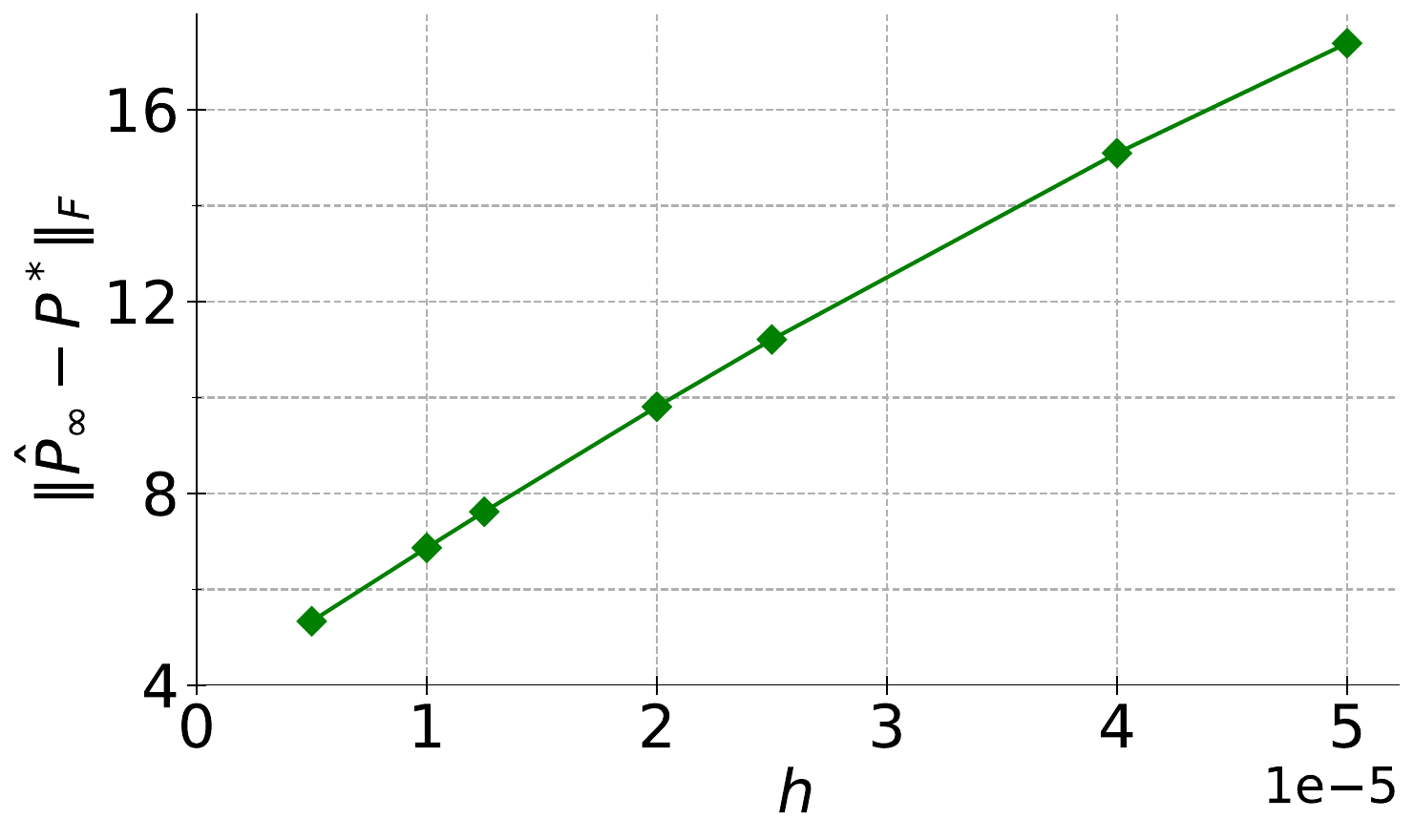}
\label{subfig:pnorm}
\vspace{-0.7cm}
    \end{minipage}
    \hfill
    \begin{minipage}{0.49\textwidth}
        \includegraphics[width=\textwidth]{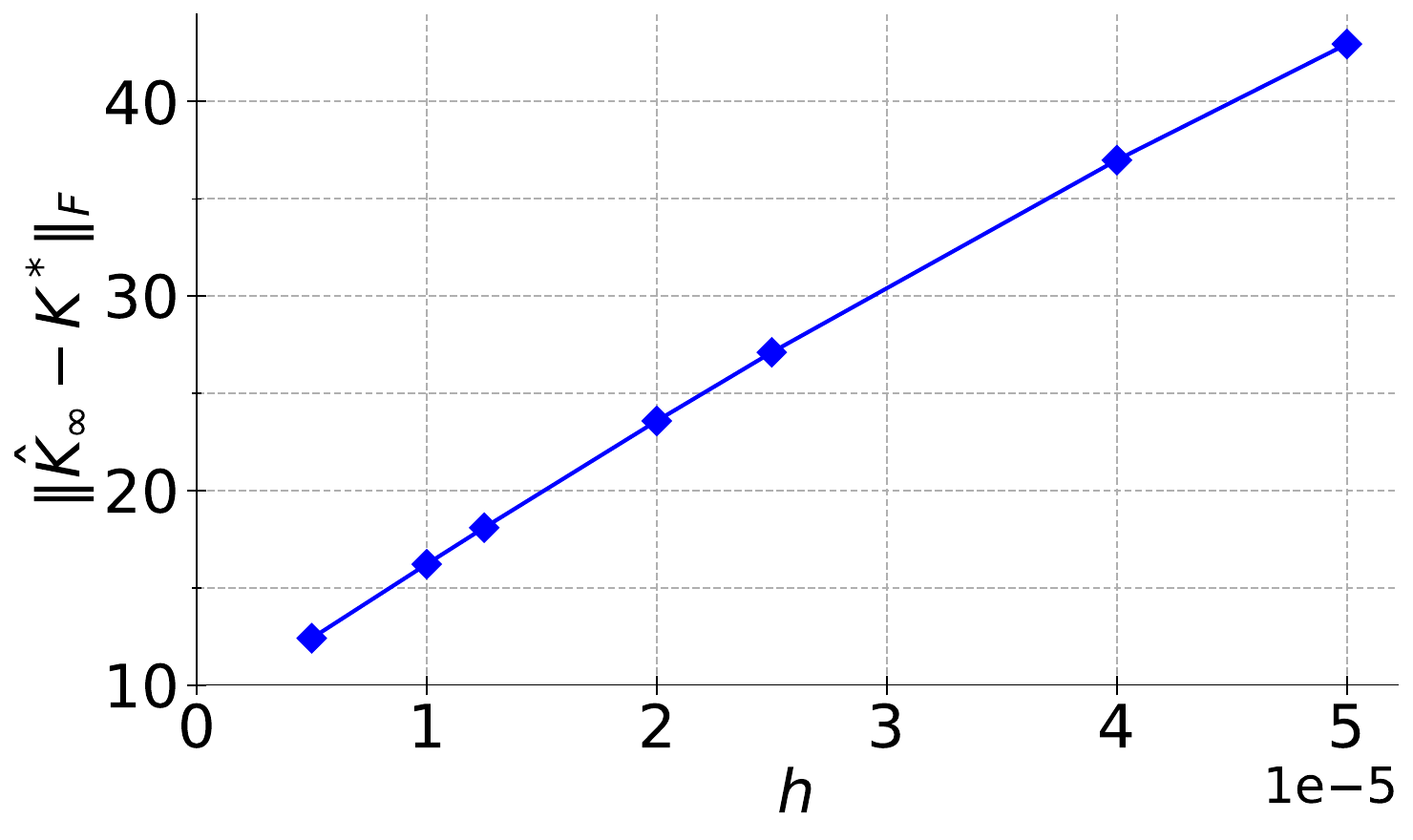}
\label{subfig:knorm}
\vspace{-0.7cm}
    \end{minipage}
\caption{Convergence of PI solutions with varying sampling periods \( h \). Left: $\|\hat{P}_{\infty} - P^*\|_F$ demonstrates a convergence rate of $O(h)$. Right: $\|\hat{K}_{\infty} - K^*\|_F$ demonstrates a convergence rate of $O(h)$.}
    \label{fig.pnormandknorm}
\end{figure}

In Figure \ref{fig.pnormandknorm}, we observe that as the sampling period \( h \) decreases, the error norms for the value matrix \(\hat{P}_{\infty}\) align with the expected convergence rate, consistent with the results of Theorem \ref{theorem.convergence of PI}. The slight deviation from a perfectly linear trend can be attributed to the fact that the theorem provides an upper bound for the error. Therefore, the actual error is smaller, resulting in a true convergence rate that is faster than \( O(h) \). 
Additionally, we are happy to find that although our theorem primarily discusses the convergence of the value matrix $\hat{P}$, our simulations have revealed that the convergence rate of the control gain matrix $\hat{K}$ also follows the pattern of $O(h)$.

\section{Conclusion and Discussion}\label{sec.conclusion}
In conclusion, our study underscores the sampling period's significant influence on ADP control performance in stochastic LQR problems. We identified that computational errors during PI crucially impact the algorithm's convergence and control policy learning. By drawing parallels between PI and Newton's method to solve the Ricatti equation, we showed that computational errors introduce additional error terms in each step of Newton's method, resulting in a convergence rate of \(O(h)\) where \(h\) represents the sampling period. 

A limitation of this work is that ADP algorithms for continuous-time stochastic systems remain largely unexplored and are generally confined to linear systems \citep{jiang2011approximate, jiang2014adaptive, bian2016adaptive, bian2018stochastic, bian2019continuous}. Consequently, our analysis is limited to linear systems, and extending it to nonlinear stochastic systems poses a considerable challenge. Besides, a technical limitation is that we assume the matrices $\Theta$ and $\hat{\Theta}$ 
are square. For future work, we will explore solutions for both under-determined and over-determined systems.
\newpage
\acks{We thank Dr. Yu Jiang from ClearMotion Inc and Prof. Zhong-Ping Jiang from New York University for sharing the code of \cite{bian2016adaptive} and for helpful discussions.}

\bibliography{ref}

\begin{thebibliography}{31}
\providecommand{\natexlab}[1]{#1}
\providecommand{\url}[1]{\texttt{#1}}
\expandafter\ifx\csname urlstyle\endcsname\relax
  \providecommand{\doi}[1]{doi: #1}\else
  \providecommand{\doi}{doi: \begingroup \urlstyle{rm}\Url}\fi

\bibitem[Allen(2007)]{allen2007modeling}
Edward Allen.
\newblock \emph{Modeling with It{\^o} stochastic differential equations}, volume~22.
\newblock Springer Science \& Business Media, 2007.

\bibitem[Bally and Talay(1995)]{bally1995euler}
Vlad Bally and Denis Talay.
\newblock The euler scheme for stochastic differential equations: error analysis with malliavin calculus.
\newblock \emph{Mathematics and computers in simulation}, 38\penalty0 (1-3):\penalty0 35--41, 1995.

\bibitem[Bally and Talay(1996{\natexlab{a}})]{bally1996law1}
Vlad Bally and Denis Talay.
\newblock The law of the euler scheme for stochastic differential equations: I. convergence rate of the distribution function.
\newblock \emph{Probability theory and related fields}, 104:\penalty0 43--60, 1996{\natexlab{a}}.

\bibitem[Bally and Talay(1996{\natexlab{b}})]{bally1996law2}
Vlad Bally and Denis Talay.
\newblock The law of the euler scheme for stochastic differential equations: Ii. convergence rate of the density.
\newblock 1996{\natexlab{b}}.

\bibitem[Bian and Jiang(2018)]{bian2018stochastic}
Tao Bian and Zhong-Ping Jiang.
\newblock Stochastic and adaptive optimal control of uncertain interconnected systems: A data-driven approach.
\newblock \emph{Systems \& Control Letters}, 115:\penalty0 48--54, 2018.

\bibitem[Bian and Jiang(2019)]{bian2019continuous}
Tao Bian and Zhong-Ping Jiang.
\newblock Continuous-time robust dynamic programming.
\newblock \emph{SIAM Journal on Control and Optimization}, 57\penalty0 (6):\penalty0 4150--4174, 2019.

\bibitem[Bian et~al.(2016)Bian, Jiang, and Jiang]{bian2016adaptive}
Tao Bian, Yu~Jiang, and Zhong-Ping Jiang.
\newblock Adaptive dynamic programming for stochastic systems with state and control dependent noise.
\newblock \emph{IEEE Transactions on Automatic control}, 61\penalty0 (12):\penalty0 4170--4175, 2016.

\bibitem[Browning et~al.(2020)Browning, Warne, Burrage, Baker, and Simpson]{browning2020identifiability}
Alexander~P Browning, David~J Warne, Kevin Burrage, Ruth~E Baker, and Matthew~J Simpson.
\newblock Identifiability analysis for stochastic differential equation models in systems biology.
\newblock \emph{Journal of the Royal Society Interface}, 17\penalty0 (173):\penalty0 20200652, 2020.

\bibitem[Einstein et~al.(1905)]{einstein1905motion}
Albert Einstein et~al.
\newblock On the motion of small particles suspended in liquids at rest required by the molecular-kinetic theory of heat.
\newblock \emph{Annalen der physik}, 17\penalty0 (549-560):\penalty0 208, 1905.

\bibitem[Gray et~al.(2011)Gray, Greenhalgh, Hu, Mao, and Pan]{gray2011stochastic}
Alison Gray, David Greenhalgh, Liangjian Hu, Xuerong Mao, and Jiafeng Pan.
\newblock A stochastic differential equation sis epidemic model.
\newblock \emph{SIAM Journal on Applied Mathematics}, 71\penalty0 (3):\penalty0 876--902, 2011.

\bibitem[Grimmett and Stirzaker(2020)]{grimmett2020probability}
Geoffrey Grimmett and David Stirzaker.
\newblock \emph{Probability and random processes}.
\newblock Oxford university press, 2020.

\bibitem[Hager(1989)]{hager1989updating}
William~W Hager.
\newblock Updating the inverse of a matrix.
\newblock \emph{SIAM review}, 31\penalty0 (2):\penalty0 221--239, 1989.

\bibitem[Jiang and Jiang(2011)]{jiang2011approximate}
Yu~Jiang and Zhong-Ping Jiang.
\newblock Approximate dynamic programming for optimal stationary control with control-dependent noise.
\newblock \emph{IEEE Transactions on Neural Networks}, 22\penalty0 (12):\penalty0 2392--2398, 2011.

\bibitem[Jiang and Jiang(2014)]{jiang2014adaptive}
Yu~Jiang and Zhong-Ping Jiang.
\newblock Adaptive dynamic programming as a theory of sensorimotor control.
\newblock \emph{Biological cybernetics}, 108\penalty0 (4):\penalty0 459--473, 2014.

\bibitem[Kantorovich and Akilov(2016)]{kantorovich2016functional}
Leonid~Vitalevich Kantorovich and Gleb~Pavlovich Akilov.
\newblock \emph{Functional analysis}.
\newblock Elsevier, 2016.

\bibitem[Kleinman(1969)]{kleinman1969optimal}
David Kleinman.
\newblock Optimal stationary control of linear systems with control-dependent noise.
\newblock \emph{IEEE Transactions on Automatic Control}, 14\penalty0 (6):\penalty0 673--677, 1969.

\bibitem[Kloeden and Platen(1992)]{kloeden1992stochastic}
Peter~E Kloeden and Eckhard Platen.
\newblock \emph{Stochastic differential equations}.
\newblock Springer, 1992.

\bibitem[Lancaster(1966)]{lancaster1966error}
Peter Lancaster.
\newblock Error analysis for the newton-raphson method.
\newblock \emph{Numerische Mathematik}, 9\penalty0 (1):\penalty0 55--68, 1966.

\bibitem[Liu and Todorov(2007)]{liu2007evidence}
Dan Liu and Emanuel Todorov.
\newblock Evidence for the flexible sensorimotor strategies predicted by optimal feedback control.
\newblock \emph{Journal of Neuroscience}, 27\penalty0 (35):\penalty0 9354--9368, 2007.

\bibitem[McLane(1971)]{mclane1971optimal}
Peter McLane.
\newblock Optimal stochastic control of linear systems with state-and control-dependent disturbances.
\newblock \emph{IEEE Transactions on Automatic Control}, 16\penalty0 (6):\penalty0 793--798, 1971.

\bibitem[Milano and Z{\'a}rate-Mi{\~n}ano(2013)]{milano2013systematic}
Federico Milano and Rafael Z{\'a}rate-Mi{\~n}ano.
\newblock A systematic method to model power systems as stochastic differential algebraic equations.
\newblock \emph{IEEE Transactions on Power Systems}, 28\penalty0 (4):\penalty0 4537--4544, 2013.

\bibitem[OpenAI(2023)]{OpenAI2023GPT4TR}
OpenAI.
\newblock Gpt-4 technical report.
\newblock \emph{ArXiv}, abs/2303.08774, 2023.
\newblock URL \url{https://api.semanticscholar.org/CorpusID:257532815}.

\bibitem[Oravecz et~al.(2011)Oravecz, Tuerlinckx, and Vandekerckhove]{oravecz2011hierarchical}
Zita Oravecz, Francis Tuerlinckx, and Joachim Vandekerckhove.
\newblock A hierarchical latent stochastic differential equation model for affective dynamics.
\newblock \emph{Psychological methods}, 16\penalty0 (4):\penalty0 468, 2011.

\bibitem[Silver et~al.(2016)Silver, Huang, Maddison, Guez, Sifre, Van Den~Driessche, Schrittwieser, Antonoglou, Panneershelvam, Lanctot, et~al.]{silver2016mastering}
David Silver, Aja Huang, Chris~J Maddison, Arthur Guez, Laurent Sifre, George Van Den~Driessche, Julian Schrittwieser, Ioannis Antonoglou, Veda Panneershelvam, Marc Lanctot, et~al.
\newblock Mastering the game of go with deep neural networks and tree search.
\newblock \emph{Nature}, 529\penalty0 (7587):\penalty0 484--489, 2016.

\bibitem[Stager and Tanner(2016)]{stager2016stochastic}
Adam Stager and Herbert~G Tanner.
\newblock Stochastic behavior of robots that navigate by interacting with their environment.
\newblock In \emph{2016 IEEE 55th Conference on Decision and Control (CDC)}, pages 6871--6876. IEEE, 2016.

\bibitem[Sz{\'e}kely~Jr and Burrage(2014)]{szekely2014stochastic}
Tam{\'a}s Sz{\'e}kely~Jr and Kevin Burrage.
\newblock Stochastic simulation in systems biology.
\newblock \emph{Computational and structural biotechnology journal}, 12\penalty0 (20-21):\penalty0 14--25, 2014.

\bibitem[Urabe(1956)]{urabe1956convergence}
Minoru Urabe.
\newblock Convergence of numerical iteration in solution of equations.
\newblock \emph{Journal of Science of the Hiroshima University, Series A (Mathematics, Physics, Chemistry)}, 19\penalty0 (3):\penalty0 479--489, 1956.

\bibitem[Wang and Zhou(2020)]{wang2020continuous}
Haoran Wang and Xun~Yu Zhou.
\newblock Continuous-time mean--variance portfolio selection: A reinforcement learning framework.
\newblock \emph{Mathematical Finance}, 30\penalty0 (4):\penalty0 1273--1308, 2020.

\bibitem[Wei et~al.(2023)Wei, Zhou, Lu, Liu, Su, and Xiao]{wei2023continuous}
Qinglai Wei, Tianmin Zhou, Jingwei Lu, Yu~Liu, Shuai Su, and Jun Xiao.
\newblock Continuous-time stochastic policy iteration of adaptive dynamic programming.
\newblock \emph{IEEE Transactions on Systems, Man, and Cybernetics: Systems}, 2023.

\bibitem[Willems and Willems(1976)]{willems1976feedback}
Jacques~L Willems and Jan~C Willems.
\newblock Feedback stabilizability for stochastic systems with state and control dependent noise.
\newblock \emph{Automatica}, 12\penalty0 (3):\penalty0 277--283, 1976.

\bibitem[Yosida(2012)]{yosida2012functional}
K{\"o}saku Yosida.
\newblock \emph{Functional analysis}.
\newblock Springer Science \& Business Media, 2012.

\end{thebibliography}

\end{document}